\documentclass[11pt]{amsart}
\usepackage{latexsym}
\usepackage{amsfonts}

\usepackage{amsmath,amsthm,amssymb}

\newtheorem{question}{Question}

\newtheorem{thm}{Theorem}[section]
\newtheorem{lem}[thm]{Lemma}
\newtheorem{cor}[thm]{Corollary}

\newtheorem{prop}[thm]{Proposition}
\theoremstyle{definition}
\newtheorem{defn}[thm]{Definition}

\newtheorem{obs}[thm]{Observation}

\newcommand{\comment}[1]{}

\begin{document} 

\author[Carl G. Jockusch, Jr.]{Carl G. Jockusch, Jr.}

\address{\tt Department of Mathematics, University of Illinois at
  Urbana-Champaign, 1409 West Green Street, Urbana, IL 61801, USA
  \newline http://www.math.uiuc.edu/\~{}jockusch/} \email{\tt
  jockusch@math.uiuc.edu}

\author[Paul Schupp]{Paul Schupp}

\address{\tt Department of Mathematics, University of Illinois at
  Urbana-Champaign, 1409 West Green Street, Urbana, IL 61801, USA}
  \email{\tt schupp@math.uiuc.edu}

\title[Generic computability,  Turing degrees
and asymptotic density]{Generic computability,
Turing degrees, and asymptotic density}

\begin{abstract}
  Generic decidability has been extensively studied in group theory,
  and we now study it in the context of classical computability
  theory.  A set $A$ of natural numbers is called \emph{generically
    computable} if there is a partial computable function which agrees
  with the characteristic function of $A$ on its domain $D$, and
  furthermore $D$ has density $1$, i.e. $\lim_{n \to \infty} |\{k<n 
  :  k \in D\}| / n = 1$.  A set $A$ is called \emph{coarsely
    computable} if there is a computable set $R$ such that the
  symmetric difference of $A$ and $R$ has density $0$.  We prove that
  there is a c.e.~set which is generically computable but not coarsely
  computable and vice versa.  We show that every nonzero Turing degree
  contains a set which is not generically computable and also a set
  which is not coarsely computable.  We prove that there is a c.e.~set
  of density $1$ which has no computable subset of density $1$.
  Finally, we define and study generic reducibility.
\end{abstract}


\maketitle


\section{Introduction}\label{intro}

In recent years there has been a general realization that worst-case
complexity measures such as \emph{P, NP}, exponential time, and just
being computable often do not give a good overall picture of the
difficulty of a problem.  The most famous example of this is the
Simplex Algorithm for linear programming, which runs hundreds of times
every day, always very quickly.  Klee and Minty \cite{Klee} constructed
examples for which the simplex algorithm takes exponential time, but
these examples do not occur in practice.

Gurevich \cite{Gurevich} and Levin\cite{Levin} independently
introduced the idea of \emph{average-case complexity}.  Here one has a
probability measure on the instances of a problem and one averages the
time complexity over all instances.  An important result is the result
of Blass and Gurevich \cite{BG} that the \emph{Bounded Product
  Problem} for the modular group, $PSL(2,\mathbb{Z})$, is
\emph{NP}-complete but has polynomial time average-case complexity.
Average-case complexity is, however, difficult to work with because it
is highly sensitive to the probability distribution used and one must
still consider all cases.

  \emph{Generic-case complexity}
was introduced by Kapovich, Miasnikov, Schupp and Shpilrain \cite{KMSS}
as a complexity measure  which is much easier to work with.  The basic
idea is that one considers partial algorithms which give no incorrect answers
and  fail to converge only on a
``negligible'' set of inputs as defined below.

\begin{defn}[Asymptotic density]

Let $\Sigma$ be a nonempty finite alphabet  and let
$\Sigma^*$ denote the set of all finite words on  $\Sigma$.
The \emph{length}, $|w|$,  of a word  $w$  is the number of
letters in $w$.
Let $S$ be a subset of
$\Sigma^*$. For every $n\ge 0$ let $S\lceil n$ denote  the set of all  words in
$S$  of length  at most $n$.  Let

     \[ \rho_n(S) = \frac{|S \lceil n|}{|\Sigma^* \lceil n|} \]

We define the \emph{upper density} $\overline{\rho}(S)$ of $S$ in $\Sigma^*$
as

\[
\ \overline{\rho}(S) := \limsup_{n \to \infty} \rho_n(S)
\]

Similarly, we define the \emph{lower density} $\underline{\rho}(S)$ of $S$ in $\Sigma^*$
as

\[
 \underline{\rho}(S) := \liminf_{n \to \infty} \rho_n(S)
\]

If the actual limit
\[ 
\rho (S) =lim_{n \to \infty} \rho_n(S) \mbox{\  exists, then \  } \rho (S) \mbox{\  is the \emph{(asymptotic) density}
of \ } S \mbox{\  in \ }  \Sigma^*.
\]

\end{defn}

\begin{defn} A subset $S$ of $\Sigma^*$ is \emph{generic} if $\rho (S) = 1$
and $S$ is \emph{negligible} if  $\rho (S) = 0$.
\end{defn}

It is clear that $S$ is generic if and only if its complement
$\overline{S}$ is negligible. Also, the union and intersection of a
finite number of generic (negligible) sets is generic (negligible).

\begin{defn}
In the case where the limit
\[ lim_{n \to \infty} ~ {\rho_n(S)} = \rho(S) = 1 \] we are sometimes
interested in estimating the speed of convergence of the sequence
$\{\rho_n(S)\}$.  To this end, we say that the convergence is
\emph{exponentially fast} if there are $0 < \sigma <1$ and $C>0$ such
that for every $n\ge 1$ we have $1-\rho_n(S) \le C \sigma^n$.  In this
case we say that $S$ is \emph{strongly generic}.
\end{defn}

\begin{defn} Let $S$ be a subset of $\Sigma^*$ with characteristic
  function $\chi_S$.  A partial function $\Phi$ from $\Sigma^*$ to
  $\{0,1\}$ is called a \emph{generic description} of $S$ if $\Phi (x)
  = \chi_S (x)$ whenever $\Phi(x)$ is defined
  (written $\Phi(x) \downarrow$) and  the domain of $\Phi$ is
  generic in $\Sigma^*$.  A set $S$ is called \emph{generically
    computable} if there exists a \emph{partial computable} function $\Phi$
  which is a generic description of $S$.  We stress that \emph{all}
  answers given by $\Phi$ must be correct even though $\Phi$ need not
  be everywhere defined, and, indeed, we do not require the domain of
  $\Phi$ to be computable.

\end{defn}

It turns out that one can prove sharp results about generic-case
complexity without even knowing the worst-case complexity of a
problem.  Magnus \cite{L-S,Magnus} proved that one-relator groups have
solvable word problem in the 1930's.  We do not know any precise bound
on complexity over the entire class of one-relator groups.  But for
any one-relator group with at least three generators, the word problem
is strongly generically linear time \cite{KMSS}.  Also, we do not know
whether or not the isomorphism problem restricted to one-relator
presentations is solvable but the problem is strongly generically
linear time \cite{KS1, KSS}.  A very clear discussion of Boone's group
with unsolvable word problem is given in Rotman \cite{Rotman}.  The
proof shows that one can model a universal Turing machine inside the
group and words coding the Turing machine are called ``special''
words. Such words are indeed very special and the word problem
for Boone's group is strongly generically linear time \cite{KMSS}.
Indeed, it is not known whether there is a finitely generated group whose 
word problem is not generically computable.

   In many ways, generic computability is orthogonal to the idea of
Turing degrees since generic computability depends on how information
is distributed in a given set.

\begin{obs} Every Turing degree contains a set which is strongly
  generically computable in linear time.  Let $A$ be an arbitrary
  subset of $ \omega$ and let $S \subseteq \{0,1\}^*$ be the set
  $\{ 0^n : n \in A\}$.  Now $S$ is Turing equivalent to $A$ and is
  strongly generically computable in linear time by the algorithm
  $\Phi$ which, on input $w$, answers ``No'' if $w$ contains a $1$ and
  does not answer otherwise.  Here all computational difficulty is
  concentrated in a negligible set, namely the set of words containing
  only $0$'s.   Note that since the algorithm given is independent of the
  set $A$, the observation shows that one algorithm can generically 
  decide uncountably many sets.
\end{obs}

The next observation is a general abstract version of Miasnikov's and
Rybalov's proof (\cite{MR}, Theorem 6.5) that there is a finitely
presented semigroup whose word problem is not generically computable.

\begin{obs}\label{nzd} Every nonzero Turing degree contains a set which is not
  generically computable.  Let $A$ be any noncomputable subset of
  $\omega$ and let $T \subseteq \{0,1\}^*$ be the set $\{ 0^n 1 w
  : n \in A, w \in \{0,1\}^* \}$.  Clearly $A$ and $T$ are Turing
  equivalent.  For a fixed $n_0$, $\rho( \{0^{n_0} 1 w : w \in
  \{0,1\}^* \} ) = 2^{-(n_0 + 1)} > 0$.  A generic algorithm for a set
  must give an answer on some members of any set of positive density.
  Thus $T$ cannot be generically computable since if $\Phi$ were a
  generic algorithm for $T$ we could just run bounded simulation of
  $\Phi$ on the set $\{0^{n} 1 w : w \in \{0,1\}^* \}$ until $\Phi$
  gave an answer, thus deciding whether or not $n \in A$.  Here the idea is
  that the single bit of information $\chi_A(n)$ is ``spread out'' to
  a set of positive density in the definition of $T$.  Also note
  that if $A$ is c.e. then $T$ is also c.e.~and thus every nonzero
  c.e.~Turing degree contains a c.e.~set which is not generically computable.
\end{obs}

In the current paper we study generic computability for sets of
natural numbers using the concepts and techniques of computability
theory and the classic notion of asymptotic density for sets of
natural numbers.  An easy result, analogous to  Observation
\ref{nzd} above, is that every nonzero Turing degree contains a set of natural
numbers which is not generically computable.

 We define the  notion
of being \emph{densely approximable} by a class $\mathcal{C}$ of sets
and observe that a set $A$ is generically computable if and only if
it is densely approximable by c.e.~(computably enumerable)  sets.
 We prove that there is a c.e.~set of
density $1$ which has no computable subset of density $1$. 
It follows as a corollary  that there is a generically computable set $A$
such that no generic algorithm for $A$ has a computable domain.

We call a set $A$ of natural numbers \emph{coarsely computable} if
there is a computable set $B$ such that the symmetric difference of
$A$ and $B$ has density $0$.  We show that there are c.e.~sets which
are coarsely computable but not generically computable and c.e.~sets
which are generically computable but not coarsely computable.  We also
prove that every nonzero Turing degree contains a set which is not
coarsely computable.

 We consider a relativized notion of generic
computability  and also introduce a notion of generic reducibility which
 which gives a degree structure and which
is  related to enumeration  reducibility.  
Almost all of our proofs
use the  collection of sets $\{ R_n \}$ defined below which form a 
partition of $\mathbb{N} - \{ 0 \}$ into subsets of positive density.
We use this collection  to define a natural embedding of the Turing degrees 
into the generic degrees and show that this embedding is proper.  We close by
describing some related ongoing work with Rod Downey and stating  some
open questions.

\section{Generic computability of subsets of $\omega$}

    We identify the set $\mathbb{N} = \{0,1,\dots\}$ of natural numbers
with the set $\omega$ of finite ordinals and from now on we will focus on generic computability properties of
subsets of $\omega$ and how these interact with some classic
concepts of computability theory.  Thus, we are using the $1$-element
alphabet $\Sigma = \{1\}$ and identifying $n \in \omega$ with its
unary representation $1^n \in \{1\}^*$, so that we also identify
$\omega$ with $\{1\}^*$.  In this context, of course, our
definition of (upper and lower) density for subsets of $\{1\}^*$
agrees with the corresponding classical definitions for subsets of
$\omega$.  In particular, the density of $A$, denoted $\rho(A)$ is
given by $\lim_n \frac{|A \lceil n|}{n+1}$, provided this limit exists,
where $A \lceil n = A \cap [0,n]$.  Further, for $A \subseteq B$, the
density of $A$ in $B$ is $\lim_n \frac{|A \lceil n|}{|B \lceil n|}$,
provided $B$ is nonempty and this limit exists.  Corresponding
definitions hold for upper and lower density.  It is clear that if $A$
has positive upper density in $B$, and $B$ has positive density, then
$A$ has positive upper density.

Our notation for computability is mostly standard, except that we use
$\Phi_e$ for the unary partial function computed by the $e$-th Turing
machine, and we let $\Phi_{e,s}$ be the part of $\Phi_e$ computed in at 
most $s$ steps.  Let $W_e$ be the domain of $\Phi_e$.   We identify
a set $A \subseteq \omega$ with its characteristic function $\chi_A$.

\begin{defn} Let $\mathcal{C}$ be a family of subsets of $\omega$.  A set
$A \subseteq \omega$ is \emph{densely} $\mathcal{C}$-\emph{approximable} if
there exist sets
\[ C_0, C_1 \in \mathcal{C} \mbox{\  such that \ }
 C_0 \subseteq \overline{A}, \ C_1 \subseteq A \mbox{\  and \ }   C_0 \cup C_1 
\mbox{ \  has density \ } 1. \]
\end{defn}

The following proposition corresponds to the basic fact that a set
$A$ is computable if and only if both $A$ and its complement $\overline A$
are computably enumerable.

\begin{prop} \label {approx} A set $A$ is generically computable if
  and only if $A$ is densely approximable by c.e.~sets.
\end{prop}

\begin{proof} If $A$ is densely approximable by c.e.~sets then
there exist c.e.~sets  $C_0 \subseteq \overline{A}$ and
 $C_1 \subseteq A$ such that $C_0 \cup C_1$ has density $1$.
  For a given $x$, start enumerating  both
  $C_0$ and $C_1$   and if $x$ appears, answer accordingly.

    If $A$ is generically computable by a partial computable function $\Phi$,
then the sets $C_0$ and $C_1$ on which $\Phi$ respectively
answers ``No'' and ``Yes'' are the desired c.e.~sets.
\end{proof}

\begin{cor}  Every c.e.~set of density $1$ is generically computable.
\end{cor}

Recall that a set $A$ is \emph{immune} if $A$ is infinite and $A$ does
not contain any infinite c.e.~set and $A$ is \emph{bi-immune} if both
$A$ and its complement $\overline{A}$ are immune.  If the union
 $C_0 \cup C_1$ of two c.e.~sets has density $1$, certainly at least one
of them is infinite. Thus we have the following corollary.    

\begin{cor} \label{bi-immune} No bi-immune set is generically computable
\end{cor}

Now the class of bi-immune sets is both comeager and of measure $1$.
(This is clear by countable additivity since the family of sets
containing a given infinite set is of measure $0$ and nowhere dense.)
Thus the family of generically computable sets is both meager and of
measure $0$.  See Cooper \cite{Cooper} for the definition of
$1$-\emph{generic} in computability theory and see Nies \cite{Nies} for
the definition of $1$-\emph{random}.  (This use of the word
``generic'' in the term ``$1$-generic'' has no relation to our general use of
``generic'' throughout this paper.)  We cite here only the facts that
 $1$-generic sets and and  $1$-random sets  are  bi-immune, and it
thus follows that no generically computable set is $1$-generic or
$1$-random.

  The following sets $R_{k}$ play a crucial role in almost all of our proofs.

\begin{defn}
\[ R_k = \{ m :  \ 2^k | m, \ 2^{(k + 1)} \nmid m \}  \]
\end{defn}

For example, $R_0$ is the set of odd nonnegative integers.  Note that
$\rho(R_k) = 2^{-(k+1)}$.  The collection of sets $\{R_k \}$ forms a
partition of unity for $\omega - \{0\}$ since these sets are pairwise
disjoint and $ \bigcup_{k = 0}^{\infty} R_k = \omega - \{0\}$.

 From the definition of asymptotic density it is clear that we have
 \emph{finite additivity} for densities: If $S_1, \dots ,S_t$ are pairwise
 disjoint sets whose densities exist, then

\[  \rho(\bigcup_{i = 1}^{t} S_i )  =  \sum_{i=1}^{t} \rho(S_i) .  \]

Of course, we do not have general countable additivity for densities, since
$\omega$ is a countable union of singletons.  However, we do have
countable additivity in certain restricted situations, where the
``tails'' of a sequence contribute vanishingly small density to the
union of a sequence of sets.

\begin{lem}[Restricted countable additivity] \label{rca}
  If $\{ S_i \}, i = 0, 1, \dots $ is a countable collection of pairwise
  disjoint subsets of $\omega$ such that each $\rho(S_i)$ exists
  and $\overline{\rho}( \bigcup_{i = N}^{\infty} S_i ) \to 0$ as $N
  \to \infty$, then

\[   \rho (\bigcup_{i = 0}^{\infty} S_i )  =  \sum_{i=0}^{\infty} \rho(S_i) .  \]

\end{lem}

\begin{proof}
  The sequence of partial sums $ \sum_{i=0}^{t} \rho(S_i) $ is a
  monotone nondecreasing sequence bounded above by $1$, and so
  converges.  Let its limit be $r$.  Now

  \[ \frac{| (\bigcup_{i = 0}^{\infty} S_i ) \lceil n |}{n+1} = \frac{|
    (\bigcup_{i = 0}^{N} S_i ) \lceil n |}{n+1} + \frac{| (\bigcup_{i =
      N+1}^{\infty} S_i ) \lceil n|}{n+1}
\]
We need to show that the term on the left approaches $r$ as $n \to
\infty$.  For any $N$, as $n \to \infty$ the first term on the right
approaches $\sum_{i=0}^{N} \rho(S_i)$ by finite additivity and thus
approaches $r$ as $N \to \infty$.  We are done because, by hypothesis,
the second term on the right can be made arbitrarily close to $0$ by
choosing $N$ sufficiently large and then $n$ sufficiently large.  In
more detail, let $\epsilon > 0$ be given.  Choose $N_0$ so that for
all $N \geq N_0$, $\overline{\rho}( \bigcup_{i = N}^{\infty} S_i ) <
\epsilon/3$.  Choose $N_1$ so that for all $N \geq N_1$, $|r - \sum_{i
  = 0}^N \rho(S_i)| < \epsilon/3$.  Fix $N = \max\{N_0, N_1\}$.
Rewrite the displayed equation above as $a_n = b_{n,N} + c_{n,N}$, so
that we are trying to prove that $a_n \to r$ as $n \to \infty$.
Choose $n_0$ so large so that for all $n \geq n_0$, $c_{n,N} <
\epsilon/3$.  Choose $n_1$ so large that for all $n \geq n_1$,
$|\sum_{i = 0}^N \rho(S_i) - b_{n,N}| < \epsilon / 3$.  Standard
calculations show that if $n \geq \max \{n_0, n_1\}$, then $|a_n - r|
< \epsilon$.
\end{proof}

\begin{defn} If $A \subseteq \omega$ then $\mathcal{R}(A) =
  \bigcup_{n \in A} R_n$
\end{defn}

Our sequence $\{R_n\}$ clearly satisfies the hypotheses of Lemma
\ref{rca}, so we have the following corollary.

\begin{cor} $\rho(\mathcal{R}(A)) = \sum_{n \in A} 2^{-(n+1)}$
\end{cor}

This gives an explicit construction of sets  with a pre-assigned densities.

\begin{cor} Every real number $r \in [0,1]$ is a density.
\end{cor}

If $r = . b_0 b_1 b_2 \dotsb_i \dots$ is the binary expansion of $r$, let
$A = \{i : b_i = 1\}$ and then $\rho(\mathcal{R}(A)) = r$.
Recall that a real number $r$ is \emph{computable} if and only if
 there is a computable function $f : \mathbb{N} \to \mathbb{Q}$ 
such that  $| r - f(n) | \le 2^{-n}$ for all $n \ge 0$.

\begin{obs} The density $r_A$ of $\mathcal{R}(A)$, i.e.  $\sum_{n \in
    A} \rho(R_n)$, is a computable real if and only if $A$ is
  computable.
\end{obs}

\begin{proof} If $A$ is computable, to compute the first $t$ bits of
  $r_A$, check if $0,\dots, t$ are in $A$, and take the resulting
  fraction $.b_0 \dots b_t$.  If  $r_A$ is computable then
  there exists an algorithm $\Phi$ which, when given $t$, computes the
  first $t$ digits of the binary expansion of $r_A$.  To see if $n \in
  A$, compute the first $(n+1)$ bits of $r_A$.
\end{proof}

We shall  later characterize those reals which are densities of computable sets.

It is obvious that every Turing degree contains a generically
computable subset of $\omega$.  Namely, given a set $A$, let $B =
\{2^n : n \in A\}$.  Then $B$ is generically computable via the
algorithm which answers ``no'' on all arguments which are not powers
of $2$ and gives no answer otherwise.   Now  $A$ and $B$ are Turing equivalent, and in fact they are many-one
equivalent if $A \neq \omega$.

\begin{obs} \label{nzd1} The set $\mathcal{R}(A)$ is Turing equivalent
  to $A$ and is generically computable if and only if $A$ is
  computable.  Hence, every nonzero Turing degree contains a subset of
  $\omega$ which is not generically computable.
\end{obs} 

This is the same argument as in Observation 1.7, namely any generic
algorithm for $\mathcal{R}(A)$ must converge on an element of each
$R_n$.  Note also that $A$ and $\mathcal{R}(A)$ are many-one
equivalent for $A \neq \omega$, so that every many-one degree which
contains a noncomputable set also contains a set which is not
generically computable.

\begin{defn} Two sets $A$ and $B$ are \emph{generically similar},
  which we denote by  $A \thicksim_g B$, if their symmetric difference $ A
  \triangle B = (A \diagdown B ) \cup ( B \diagdown A)$ has density
  $0$.
\end{defn}

It is easy to check that $\thicksim_g$ is an equivalence relation.
Any set of density $1$ is generically similar to $\omega$, and any
set of density $0$ is generically similar to $\emptyset$.

\begin{defn} A set $A$ is \emph{coarsely computable} if $A$ is
  generically similar to a computable set.
\end{defn}

From the remarks above, all sets of density $1$ or of density $0$ are
coarsely computable.  One can think of coarse computability in the
following way: The set $A$ is coarsely computable if there exists a
\emph{total} algorithm $\Phi$ which may make mistakes on membership in
$A$ but the mistakes occur only on a negligible set.  While being
coarsely computable is not of 
practical value in providing algorithms we shall see that it is
an interesting measure of how computable a set is.

\begin{obs} The word problem of any finitely generated group $G =
  \langle X:R \rangle$ is coarsely computable.
\end{obs}

\begin{proof} If $G$ is finite then the word problem is computable.
  Indeed, since computability is independent of the particular
  presentation as long as it is finitely generated, we can use the
  finite multiplication table presentation, which shows that the word
  problem is even a regular language.  Now if $G$ is an infinite
  group, the set of words on $(X \cup X^{-1})^*$ which are not equal
  to the identity in $G$ has density $1$ and hence is coarsely
  computable. (See, for example, \cite{Woess}.)
\end{proof}

\begin{prop} There is a c.e.~set which is coarsely computable but not
  generically computable.
\end{prop}

\begin{proof}
  Recall that a c.e.~set $A$ is \emph{simple} if $\overline{A}$ is
  immune.  It suffices to construct a simple set $A$ of density $0$,
  since any such set is coarsely computable but not generically
  computable by Proposition \ref{approx}.  This is done by a slight
  modification of Post's simple set construction.  Namely, for each
  $e$, enumerate $W_e$ until, if ever, a number $> e^2$ appears, and
  put the first such number into $A$.  Then $A$ is simple, and $A$ has
  density $0$ because for each $e$, it has at most $e$ elements less
  than $e^2$.
\end{proof}

It is easily seen that the family of coarsely computable sets is
meager and of measure $0$.  In fact, if $A$ is coarsely computable,
then $A$ is neither $1$-generic nor $1$-random.  To see this, note
first that if $A$ is $1$-random and $C$ is computable, then the
symmetric difference $A \triangle C$ is also $1$-random, and the
analogous fact also holds for $1$-genericity.  The result follows
because $1$-random sets have density $1/2$ (\cite{Nies}, Proposition
3.2.13) and $1$-generic sets have upper density $1$.

\begin{thm} There exists a c.e.~set which is not generically similar
  to any co-c.e.~set and hence is neither coarsely computable nor
  generically computable.
\end{thm}

\begin{proof} Let $\{ W_e \}$ be a standard enumeration of all
  c.e.~sets. Let
  \[ A = \bigcup_{e \in \omega} (W_e \cap R_e ) \] Clearly, $A$ is
  c.e. We first claim that $A$ is not generically similar to any
  co-c.e.~set and hence is not coarsely computable. Note that
  \[ R_e \subseteq A \triangle \overline{W_e} \] since if $n \in R_e$
  and $n \in A$ , then $n \in (A \diagdown \overline{W_e})$, while if
  $n \in R_e$ and $n \notin A$, then $n \in ( \overline{W_e} \diagdown
  A)$.  So, for all $e$, $(A \triangle \overline{W_e})$ has positive
  lower density, and hence $A$ is not generically similar to
  $\overline{W_e}$.  It follows that $A$ is not coarsely computable.
  Of course, this construction is simply a diagonal argument, but
  instead of using a single witness for each requirement, we use a set
  of witnesses of positive density.

  Suppose now for a contradiction that $A$ were generically
  computable.  By Proposition \ref{approx}, let $W_a$, $W_b$ be
  c.e.~sets such that $W_a \subseteq A$, $W_b \subseteq \overline{A}$,
  and $W_a \cup W_b$ has density $1$.  Then $A$ would be generically
  similar to $\overline{W_b}$ since
  $$A \triangle \overline{W_b} \subseteq \overline{W_a \cup W_b}$$
  and $\overline{W_a \cup W_b}$ has density $0$. This shows that $A$ is
  not generically computable.
\end{proof}

\begin{defn} If $A \subseteq \omega$ and $\{A_s\}$ is a sequence of
  finite sets we write $\lim_s A_s = A$, if for every $n$ we have, for
  all sufficiently large $s$, $n \in A$ if and only if $n \in A_s$.
\end{defn}

  The Limit Lemma, due to J.~Shoenfield, characterizes the sets $A$
  computable from the halting problem $0'$ as the limits of uniformly
  computable sequences of finite sets.

\begin{lem}[The Limit Lemma]
  Let $A \subseteq \omega$  Then $A \leq_T 0'$ if and only if there is
  a uniformly computable sequence of finite sets $\{A_s\}$ such that $\lim_s
A_s = A$.
\end{lem}

We note that  by Post's Theorem, the sets Turing reducible to $0'$ are precisely
the sets which are $\Delta^0_2$ in the arithmetical hierarchy.

\begin{thm} The set $\mathcal{R}(A) = \bigcup_{n \in A} R_n$ is
  coarsely computable if and only if  $A \leq_T 0'$ .
\end{thm}

\begin{proof} First suppose that $A \leq_T 0'$.  Then by the Limit
  Lemma there is a uniformly computable sequence $\{A_s\}$ of finite sets
  such that $\lim_s A_s = A$.  To construct a computable set
  $C$ generically similar to $\mathcal{R}(A)$ we
  do the following.  Any $n$ is in a unique set $R_k$. Compute this
  $k$, so $n \in \mathcal{R}(A)$ if and only if $ k \in A$.  We put $n$
  into $C$ if and only if $k$ is in the
  approximating set $A_n$.  This condition is computable.  Now note
  that if $n$ is sufficiently large then $k \in A $ if and only if $ k
  \in A_n$.  Hence
  \[ (C \triangle \mathcal{R}(A)) \cap R_k \] is
  finite for all $k$.  Let $D = (C \triangle
  \mathcal{R}(A))$.  Then $D \cap R_k$ has density $0$ for all $k$ and
  thus $D$ has density $0$ by Lemma \ref{rca} on restricted countable
  additivity.  It follows that $\mathcal{R}(A)$ is coarsely
  computable.

  Now suppose that $\mathcal{R}(A)$ is coarsely computable, that is,
  it is generically similar to a computable set $C$.  We need to show
  that  $A \leq_T 0'$ by finding a uniformly computable
  sequence of finite sets $\{A_s\}$ with $\lim_s A_s = A$.

  Note that $\rho (C) = \rho(\mathcal{R}(A))$ which exists, and
  $\rho(C \cap R_n) = \rho(\mathcal{R}(A) \cap R_n)$.  So if $n \in
  A$, then $\rho(\mathcal{R}(A) \cap R_n) = \rho(R_n) = 2^{-(n+1)}$ while if $n
  \notin A$, then $\rho( R_n \cap \mathcal{R}(A)) = 0$. Thus, we can
  use our ability to approximate $\rho(C \cap R_n)$ to approximate $A$.

      At stage $s$, for every $n \le s$, calculate
      \[ \rho_s(C \cap R_n) = \frac{ |( C \cap R_n) \lceil s | }{s+1}.  \]
      Put $n$ into $A_s$ if and only if this fraction is $\ge
      \frac{1}{2} (2^{-(n+1)})$.  The sequence
      $\{A_s \}$ is uniformly computable. It converges to $A$ since
      $\rho_s (C \cap R_n)$ converges to $\rho(C \cap R_n)$ as $s \to
      \infty$.
\end{proof}

In particular, if $A$ is any set Turing reducible to $0'$ but not
computable then $\mathcal{R}(A)$ is coarsely computable but not
generically computable.  We can now prove the following theorem.

\begin{thm} Every nonzero Turing degree contains a set which is not
coarsely computable.
\end{thm}

\begin{proof} If $A$ is not Turing reducible to $0'$, then
  $\mathcal{R}(A)$, which is Turing equivalent to $A$, is not coarsely
  computable by the previous theorem.  Now assume that $A$ is
  noncomputable and $A \leq_T 0'$.  We now apply Theorem 1.2 of
  \cite{MM} which implies that every nonzero Turing degree
  $\bf{a \le 0'}$ computes a function $f$ which is not
  majorized by any computable function.  The argument is now
  essentially a diagonalization argument using such an $f \leq_T A$ as
  a time bound.

  We now construct a set $B \leq_T f$ which is not coarsely
  computable.  This suffices since then
 \[ A \oplus B = \{ 2n: n \in A \} \cup \{ 2n+1: n \in B \} \]
  is a set Turing
  equivalent to $A$ which is not coarsely computable.

  For every pair $\langle e,k \rangle$ the requirement $P_{e,k}$ is
  that if $\Phi_e$ is a total $\{0,1\}$-valued function, then there
  exists a $j \ge k$ such that $\Phi_e$ and the characteristic
  function of $B$ disagree on all points in the interval $[2^j,
  2^{(j+1)} )$.  If all the $P_{e,k}$ are met, then the upper density
  of $B \triangle \Phi^{-1}(1)$ is at least $\frac{1}{2}$ for each
  $e$ with $\Phi_e$ total and $\{0,1\}$-valued.  Hence $B$ is not
  generically similar to the set whose characteristic function is
  $\Phi_e$.

  For definiteness, set  $0 \notin B$.  We now determine $B$ on
  each interval $[2^j, 2^{j+1})$ in the natural order.  Initially, no
  requirements are met.  Suppose that $B$ has been defined on each
  interval $[2^i, 2^{i+1})$ for $i < j$.  Say that a requirement
  $P_{e,k}$ \emph{requires attention} if it is not yet met, $j
  \geq k$, and
  \[ \Phi_{e,f(j)} (x) \downarrow \mbox{\ for all \ } x \in [2^j,
  2^{(j+1)} ) \] If there is no requirement $P_{e,k}$ with $\langle e
  , k \rangle \leq j$ which requires attention, let $B$ be empty on
the interval $[2^k , 2^{k+1})$.  Otherwise, let $\langle e, k \rangle$
be minimal such that $P_{e,k}$ requires attention.  Make $B(x)$ and
$\Phi_e (x)$ disagree on all $ x \in [2^j, 2^{(j+1)} )$, and
declare $P_{e,k}$ met (forever).  In this case we say that $P_{e,k}$
\emph{receives attention}.

We claim that all the requirements $P_{e,k}$ are satisfied.  Note that each
such requirement receives attention at most once.  Suppose for a
contradiction that $P_{e,k}$ is not met, so that $\Phi_e$ is total
and $\{0,1\}$-valued and $P_{e,k}$ never receives attention.  Since
there are only finitely many stages at which requirements $P_a$ for $a
< \langle e, k \rangle$ receive attention, there are only finitely
many $j$ such that $\Phi_{e, f(j)}$ is defined on all points in the
  interval $[2^j, 2^{j+1})$.  Let $g(j)$ be the first stage such that
  $\Phi_{e,g(j)}$ is defined for all points in $ [2^j, 2^{(j+1)} )$.
  Now $g$ is a computable function, but $g(j) \ge f(j)$ for all
  sufficiently large $j$, contradicting that $f$ is not majorized by any
  computable function.
\end{proof}

The above proof is less uniform than the proof of the corresponding
result (Observation \ref{nzd1}) for generic computability.  More precisely,
the proof of Observation \ref{nzd1}  shows that there is a fixed oracle
Turing machine $M$ such that $M^A$ is a generically noncomputable set
of the same Turing degree as $A$ for every noncomputable set $A$,
namely $M^A = \mathcal{R}(A)$.  However, we do not know whether there is
a fixed such $M$ with the corresponding property for coarse
computability.

A real number $r$ which is computable relative to $0'$ is called a
$\Delta_{2}^{0}$ real, and it is well known that these are the reals
whose binary expansion is computable from $0'$.  It then follows from
the Limit Lemma that a real number $r \in [0,1]$ is $\Delta_{2}^{0}$
if and only if $r = lim_n q_n$ for some computable sequence of
rational numbers in the interval $(0,1)$.

  \begin{thm} \label{d02} A real number $r \in [0,1]$ is the density
    of some computable set if and only if $r$ is a $\Delta_{2}^{0}$
    real.
\end{thm}

\begin{proof} If $A$ is computable  then we can compute
  \[ q_n = \rho_n(A) = \frac{|\{ k : k \le n,\ k \in A \}|}{n+1} \] for
  all $n$.  Thus, if $\rho(A) = lim_{n \to \infty} \rho_n(A)$ exists,
  its value $r$ is a $\Delta_{2}^{0}$ real.

  We must now show that if $r = lim_n q_n$ is the limit of a
  computable sequence of rationals in the interval $(0,1)$, there is a
  computable set $A$ with $\rho(A) = r$.  We define a computable
  increasing sequence $\{ s_n \}$ of positive integers such that
\[ \left|   \frac{ |A \lceil s_n | }{s_n + 1} -   q_n \right| \le \frac{1}{n}
    \mbox{ \ and \ }  lim_{n \to \infty} \frac{ |A \lceil s_n | }{s_ n+1}  = r. \]
Take $s_1 = 1$ and put $0$ in $A$.  If $A \lceil s_n $ is already defined
there are two cases.

    If $\frac{ |A \lceil s_n | }{s_n + 1} < q_{n+1} $ find the least $k$ such that
    \[ \frac{ |A \lceil s_n | + k}{s_n + k+1} \ge q_{n+1} .\] (Such a
    $k$ exists because $q_{n+1} < 1$.)  Let $s_{n+1} = s_n + k$ and let
    $A \lceil s_{n+1} = A \lceil s_n \cup \{ s_n + 1, \dots , s_n + k \}
    $.

   If $\frac{ |A \lceil s_n | }{s_n + 1} \ge q_{n+1} $ find the least $k$ such that
   \[ \frac{ |A \lceil s_n | }{s_n + k + 1} < q_{n+1} .\] Let $s_{n+1} =
   s_n + k$ and let $A \lceil s_{n+1} = A \lceil s_n $.  (We add no
   new elements to $A$.)  

   Since $s_n \ge n$ we have $|\rho_{s_n(A)}-q_n| \leq \frac{1}{n}$
   for all $n$.  It follows that $\rho_{s_n}(A)$ approaches $\lim_n q_n =
   r$ as $n \to \infty$.  Furthermore, by construction, $\rho_k(A)$ is
   monotone increasing or decreasing on each interval $(s_n,
   s_{n+1}]$, so that $\rho_k(A)$ is between $\rho_{s_n}(A)$ and
   $\rho_{s_{n+1}}(A)$ whenever $s_n < k < s_{n+1}$.  Hence $\rho_k(A) \to r$
   as $k \to \infty$, so $\rho(A) = r$.
   
\end{proof}

It is easily seen that every c.e.~set of upper density $1$ has a
computable subset of upper density $1$, and this makes it tempting to
conjecture that every c.e.~set of density $1$ has a computable subset
of density $1$.  Our next result is a refutation of this conjecture.
This result has several important corollaries, and the technique of
its proof will be used to show in Theorem \ref{d1} that there is a set
which is generically computable but not coarsely computable.

\begin{thm}\label{nosubset} There exists a c.e.~set $A$ of density $1$ which 
  has no computable subset of density $1$.
\end{thm}

\begin{proof} We will construct $A$ so that it does not contain any
  co-c.e.~subset of density $1$.  We will heavily use our partition of
  unity $\{ R_n \}$.  To ensure that $A$ has density $1$, we impose the 
following infinitary positive requirements:
\[ P_n :  R_n \subseteq^* A \]
where $B \subseteq^* A$ means that $B \diagdown A$ is finite.  These
   requirements ensure that $A$ has density $1$ because (if $0 \in A$)
\[ \overline{A} = \overline{A} \cap \bigcup_{n \in \omega} R_n  
=  \bigcup_{n \in \omega} (\overline{A} \cap  R_n) \]
and the last union has density $0$ by restricted countable additivity
(Theorem \ref{rca}).  

Let $\{ W_e \}$ be a standard enumeration of all c.e.~sets.  We must
ensure that if $\overline{W_e} \subseteq A$ (i.e. $W_e \cup A =
\omega)$, then $\overline{W_e}$ does not have density $1$,
(i.e. $W_e$ does not have density $0$).  Since $R_e$ has positive
density, it suffices to meet the following negative requirement $N_e$:
If $W_e \cup A = \omega$ then $W_e$ does not have upper density
$0$ on $R_e$.

The usefulness of the sets $R_e$ here is that the positive requirement
$P_e$ puts only elements of $R_e$ into $A$ and the negative
requirement $N_e$ keeps only elements of $R_e$ out of $A$.  Since the
$R_e$'s are pairwise disjoint, this eliminates the need for the usual
combinatorics of infinite injury constructions and indeed allows the
construction of $A$ to proceed independently on each $R_e$.  The idea
of the proof is that we can make the density of $A$ low on an interval
within $R_e$ by restraining $A$ on that interval, and at the same time
starting to put the rest of $R_e$ into $A$.  If eventually the
interval is contained in $W_e \cup A$, we have found an interval where
$W_e$ has high density and can start over with a new interval,
Otherwise, $W_e \cup A \neq \omega$, and we meet the requirement
vacuously with a finite restraint.

We construct each subset $A_e = A \cap R_e$, the $e$-th part of $A$,
in stages.   Initially, each $A_{e,0}$ is empty
 and the constraint $r(e,0)$ is the least
element of $R_e$. 

At stage $s$, check whether or not $W_{e, s+1} \cup A_{e,s}$
fills up $R_e$ below $r(e,s)$, that is, whether or not
$A_{e,s} \cup W_{e,s+1} \supseteq R_e \lceil r(e,s)$.  If not,
then $A_{e,s+1}$ is $A_{e,s} $ together with the first element of $R_e$ which
is greater than $r(e,s)$ and which is not already in $A_{e,s}$. Set
$r(e,s+1) = r(e,s)$ in this case.

If $A_{e,s} \cup W_{e,s+1} \supseteq R_e \lceil r(e,s)$, then
$A_{e,s+1}$ is $A_{e,s} \cup (R_e \lceil r(e,s))$.  Now choose $r(e,s+1)$
large enough so that $r(e,s+1) > r(e,s)$ and $A_{e,s+1}$ has density
less than or equal to $1/2$ on $R_e \lceil r(e,s+1)$.

For each $e$ there are two possibilities.  The first is that $ lim_{s}
r(e.s) = \alpha_e \in \omega $.  In this case note that all
elements of $R_e$ which are greater than $\alpha_e$ are put into $A_e =
A \cap R_e$.  Thus we indeed have $R_e \subseteq^* A$.  The negative
requirement $N_e$ is met vacuously because $W_e \cup A \neq
\omega$.

The second possibility is that $ lim_{s} r(e,s) = \infty $.  In this
case $W_e \cup A_e$ fills up arbitrary large initial intervals of
$R_e$.  So $R_e \subseteq A$ by construction and $W_e$ has positive
upper density on $R_e$ since it must supply at least $1/2$ of the
elements of arbitrarily large initial intervals of $R_e$.  Namely,
when $r(e,s)$ takes on a new value, at most half of the elements of
$R_e$ less than or equal to $r(e,s)$ are in $A$, and no elements of $R_e$
less than or equal to $r(e,s)$ enter $A$ until every number in $R_e$ less 
than or equal to $r(e,s)$ has been enumerated in $W_e \cup A$, so at least
half of these numbers have been enumerated in $W_e$.  This process occurs
for infinitely many values of $r(e,s)$.
\end{proof}

This theorem has two immediate corollaries.  The first follows from
the fact that any c.e.~set of density $1$ is generically computable.

\begin{cor} Generically computable sets need not be densely approximable
by computable sets.
\end{cor}

\begin{cor} There exists a generically computable set $A$ of density $1$
such that no generic algorithm for $A$ has computable domain.
\end{cor}

\begin{proof} Let $A$ be the c.e.~set of the Theorem above.  If $\Phi$
  were a generic algorithm for $A$ with computable domain then $\{ x |
  \Phi (X) \downarrow = 1 \}$ would be a computable subset of $A$ with
  density $1$, a contradiction.
\end{proof}

\begin{obs} A set $A$ is generically computable by a partial algorithm
  with computable domain if and only if $A$ is densely approximable by
  computable sets.
\end{obs}

\begin{thm} \label{d1} There is a generically computable c.e.~set $A$
  which is not coarsely computable.
\end{thm}

\begin{proof} The proof is similar to that of the previous theorem.
We will construct disjoint c.e.~sets $A_0, A_1$ such that

\[ A_0 \cup A_1 \mbox{ has density \ } 1 \mbox{\ and \ } A_1
\mbox{ is not coarsely computable.} \]

Note that both $A_0$ and $A_1$ are generically computable since they
are disjoint c.e.~sets and their union has density $1$.  So it will
follow that $A_1$ is generically computable but not coarsely
computable. We now have positive requirements
\[ P_e:  R_e \subseteq^* (A_0 \cup A_1) \]
and negative requirements
\[ N_e: \mbox{\ If \ } \Phi_e \mbox{ \ is total then \ } \Phi_{e}^{-1} (1) \triangle A_1
\mbox{ \ is not of density \ } 0. \]

Satisfaction of the positive requirement suffices to ensure that $A_0
\cup A_1$ has density $1$ as in the proof of Theorem \ref{d1}.  It is
clear that satisfaction of all of the negative requirements implies that $A_1$
is not coarsely computable.

We again have a restraint function $r(e,s)$.  Initially, each
$A_{e,0}$ is empty and the restraint $r(e,0)$ is the least element of
$R_e$.  At stage $s$, for each $e \leq s$, check whether
\[ \mbox{\ Domain \ } (\Phi_{e,s+1}) \supseteq R_e \lceil r(e,s) \] If
so, let $F$ be the set of elements of $R_e \lceil r(e,s)$ which are
not already in $A_0 \cup A_1$.  Put all elements of $F \cap
\Phi_e^{-1}(1)$ into $A_0$ and all other elements of $F$ into $A_1$.
Since by construction at least half of the elements of $ R_e \lceil
r(e,s)$ are in $F$, and $F \subseteq \Phi_e^{-1}(1) \triangle A_1$,
this action ensures that at least half of the elements of $ R_e \lceil
r(e,s)$ are in $\Phi_e^{-1}(1) \triangle A$.  Set
$r(e,s+1)$ to be the least element of $R_e$ such that at most half of
the elements of $R_e \lceil r(e,s+1)$ are in $A_{0,s+1} \cup
A_{1,s+1}$.

If
\[ \mbox{\ Domain \ } (\Phi_{e,s+1} )\nsupseteq R_e \lceil r(e,s) \] 
then put into $A_1$ the least element of $R_e$ which is greater than $r(e,s)$ and which is
not already in $A_1$.  Set $r(e,s+1) = r(e,s)$.

The proof that the positive requirements $P_e$ are met is exactly as
in the proof of Theorem \ref{d1}.  Hence $A_0 \cup A_1$ has density $1$.

It remains to show that each negative requirement $N_e$ is met.
Suppose that $\Phi_e$ is total.  Then by construction, there are
infinitely many $s$ with $r(e,s+1) > r(e,s)$, and so $\lim_s r(e,s) =
\infty$.  For each such $s$, the construction guarantees that at least
half of the elements of $ R_e \lceil
r(e,s)$ are in $\Phi_e^{-1}(1) \triangle
A_1$.   Thus the latter set has lower density at least $\frac{1}{2}$ on
$R_e$ and hence has positive lower density on $\omega$.

\end{proof}

\section{Relative Generic Computability}

  As almost always in computability theory, the previous results
relativize to generic computability using an arbitrary oracle.

\begin{defn} A set $B$ is \emph{generically} $A$-\emph{computable} if
  there exists a generic description $\Phi$ of $B$ which is a
  partial computable function relative to $A$.  Also, $B$ is
  \emph{coarsely} $A$-\emph{computable} if it is generically
  similar to a set computable from $A$.
\end{defn}

Using Post's Theorem, we see that a set $A$ is generically
$0^{(n)}$-computable if and only if it is densely approximable by
$\Sigma_{n+1}^{0}$ sets and $A$ is coarsely $0^{(n)}$-computable if
and only if it is generically similar to a $\Delta_{n+1}^{0}$ set.
Thus the previous results show that for every $n \ge 0$ there is a
$\Sigma_{n+1}^{0}$ set of density $1$ which is not densely
approximable by $\Delta_{n+1}^{0}$ sets.  Also, there are generically
$0^{(n)}$-computable sets which are not coarsely $0^{(n)}$-computable
and coarsely $0^{(n)}$-computable sets which are not generically
$0^{(n)}$-computable.

\begin{defn} Given a set $A$ the \emph{generic class} $\widehat{G}(A)$
  of $A$ is the family of all subsets of $\omega$ which are
  generically $A$-computable, that is, generically computable by
  oracle Turing machines with an oracle for $A$.
\end{defn}

\begin{obs} $A \le_T B$ if and only if $\widehat{G}(A) \subseteq
  \widehat{G}(B)$.
\end{obs}

\begin{proof} It is clear that if  $A  \le_T B$ then 
 $\widehat{G}(A) \subseteq  \widehat{G}(B)$.  On the other hand,
if  $\widehat{G}(A) \subseteq  \widehat{G}(B)$ then $\mathcal{R}(A)$ is
generically computable from $B$ but a generic computation of  $\mathcal{R}(A)$
allows one to compute $A$. Hence  $A  \le_T B$.
\end{proof}

So $A \equiv_{T} B$ if and only if $\widehat{G}(A) = \widehat{G}(B)$
and if $\boldsymbol{a}$ is a Turing degree then
$\widehat{G}(\boldsymbol{a})$ is a well-defined generic class.  If $A
<_{T} B$ then Observation 1.7 shows that $\widehat{G}(A)$ is strictly
contained in $ \widehat{G}(B)$.

\begin{obs} Let $( \mathbf{D}, \le_T )$ be the set of all Turing
  degrees partially ordered by Turing reducibility and let $(
  \mathbf{G}, \subseteq )$ be the family of all generic classes
  partially ordered by set inclusion.  Then the function
  $\mathfrak{A}$ from $\mathbf{D}$ to $\mathbf{G}$ defined by
  $\boldsymbol{a} \mapsto \widehat{G}(\boldsymbol{a})$ is an order
  isomorphism.
\end{obs}

\begin{proof} The remarks above show that $\mathfrak{A}$ is well-defined,
  $1-1$, and order-preserving  and it is onto by definition.
\end{proof}

Recall that a Turing degree $\boldsymbol{a}$ is \emph{minimal} if
$\boldsymbol{a} > 0$ and there is no Turing degree $\boldsymbol{b}$
with $ 0 < \boldsymbol{b} < \boldsymbol{a}$.  A theorem of Spector \cite{Cooper}
shows that  there exist uncountably many minimal Turing
degrees.  We can analogously define a generic class $\widehat{G}(A)$
to be \emph{minimal} if $\widehat{G}(A) \varsupsetneqq
\widehat{G}(\emptyset)$ and there is no generic class $\widehat{G}(B)$
with $\widehat{G}(\emptyset) \varsubsetneqq \widehat{G}(B)
\varsubsetneqq \widehat{G}(A)$.  It would seem to be 
difficult to directly construct minimal generic classes but the order
isomorphism $\mathfrak{A}$ gives the following corollary of Spector's
theorem.

\begin{cor} There are uncountably many minimal generic classes.
\end{cor}

It is important to note that relative generic computability does
\emph{not} give a notion of reducibility because it is not transitive.
It is generally false that if $A \in \widehat{G}(B)$ and $B \in
\widehat{G}(C)$ then $A \in \widehat{G}(C)$.  For example, let $A$ and
$B$ be Turing equivalent sets such that $B$ is generically computable
and $A$ is not generically computable.  (We have observed that every
nonzero Turing degree contains such sets $A$ and $B$.)  Then $A \in
\widehat{G}(B)$ and $B \in \widehat{G}(\emptyset)$, but $A \notin
\widehat{G}(\emptyset)$.  We introduce a related notion which is
transitive in the next section.

\section{Generic Reducibility}

The failure of transitivity just noted for relativized generic
computability is not surprising because the definition of $A \in
\widehat{G}(B)$ involves using a \emph{total} oracle for $B$ to
produce only a \emph{generic} computation of $A$.  This is analogous
to the failure of transitivity for the relation ``c.e.~in'', where an
oracle for $B$ is used to produce only an enumeration of $A$.  The
natural way to achieve transitivity is to have the oracle and the
output be of a similar nature.  The notion of enumeration reducibility
($\le_e$) has been well studied.  The intuitive concept of
enumeration reducibility is that $A \le_e B$ if there is a fixed  oracle
Turing machine $M$ which, given a listing of $B$ in any order on its
oracle tape, produces a listing of $A$.  From this point of view, when
the machine lists a number $n$ in $A$, it has used the membership of
$k$ in $B$ only for a finite set $D$ of values of $k$, and we can
effectively list the set of pairs $(n, D)$ for which this occurs.
This leads to a more convenient formal definition of enumeration
reducibility where we replace oracle Turing machines by c.e.~sets of
codes of such pairs.

\begin{defn} An \emph{enumeration operator} is a c.e.~set.  If $W$ is
  an enumeration operator, the elements of $W$ are viewed as coding
  pairs $\langle n, D \rangle$, where $n \in \omega$ and $D$ is a
  finite subset of $\omega$ identified with its canonical index
  $\sum_{k \in D} 2^k$.  We view $W$ as the  mapping from sets to sets 
 \[ X  \to  W(X) := \{n : (\exists D)[\langle n , D \rangle \in W \ \& \ D
  \subseteq X]\}\]
\end{defn}

We can now use enumeration operators to formally define enumeration reducibility.

\begin{defn}
$Y$ is \emph{enumeration reducible} to $X$ (written $Y \leq_e X$) if $Y = W(X)$ for some enumeration operator $W$.
\end{defn}

It is well known that the enumeration operators are closed under
composition and hence that enumeration reducibility is transitive.
Also, each enumeration operator $W$ is obviously $\subseteq$-monotone in
the sense that if $U \subseteq V$ then $W(U) \subseteq W(V)$.

We are now ready to define generic reducibility.  Recall that a
generic description of a set $A$ is a partial function $\Psi$ which
agrees with the characteristic function of $A$ on its domain and which
has a domain of density $1$.  If $\Psi$ is a partial function, let
$\gamma(\Psi) = \{ \langle a, b \rangle : \Psi(a) = b \}$, so that
$\gamma(\Psi)$ is a set of natural numbers coding the graph of $\Psi$.
A listing of the graph of a generic description of a set $A$ is called
a \emph{generic listing for } $A$.  Intuitively, the idea is that $A$
is generically reducible to $B$ if there is a fixed oracle Turing
machine $M$ which, given \emph{any} generic listing for $B$ on its
oracle tape, generically computes $A$.  It is again convenient to use
enumeration operators in the formal definition.

\begin{defn} $A$ is \emph{generically reducible} to $B$ (written $A \leq_g
  B$) if there is an enumeration operator $W$ such that, for every
  generic description $\Psi$ of $B$, $W(\gamma(\Psi)) = \gamma(\Theta)$ for
  some generic description $\Theta$ of $A$.
\end{defn} 
                                                                    
Note that $\leq_g$ is transitive because enumeration operators are
closed under composition.  (It is also easy to check transitivity
from the intuitive definition.) Thus generic reducibility  leads to a 
degree structure as usual.

\begin{defn} The sets $A$ and $B$ are \emph{generically
    interreducible}, written $A \equiv_{g} B$, if $A \leq_g B$
  and $B \leq_g A$.  The \emph{generic degree} of $A$, written
  $\deg_g(A)$, is $\{C : C \equiv_g A \}$.  Of course, the generic
  degrees are partially ordered by the ordering induced by
  $\leq_g$.
\end{defn}
 
The generic degrees have a least element ${\bf 0}_g$, and the elements
of ${\bf 0}_g$ are exactly the generically computable sets.  The
generic degrees form an upper semi-lattice, with join operation
induced by $\oplus$ where $A \oplus B = \{ 2n : n \in A \} \cup \{2n+1 : n \in B\}$.
  The following easy result gives another way in
which the generic degrees resemble the Turing degrees.

\begin{prop}
Every countable set of generic degrees has an upper bound.
\end{prop}

\begin{proof} Let sets $A_0, A_1, \dots$ be given.  We must produce a
  set $B$ with $A_n \leq_g B$ for all $n$.  Let the function $f_n :
  \omega \to R_n$ enumerate $R_n$ in increasing order and define $B =
  \cup_n f_n(A_n)$. Note that since $f_n$ is $1-1$ and the $R_n$ are
  disjoint, we have $B(f_n (x)) = A_n(x)$. To see that $A_n \leq_g B$,
  let $W$ be an enumeration operator such that $W(\gamma(\Psi)) =
  \gamma(\Psi \circ f_n)$ for every partial function $\Psi$.  We must
  show that if $\Psi$ is a generic description of $B$ then $\Psi \circ
  f_n$ is a generic description of $A_n$, First, note that if
  $\Psi(f_n(x)) \downarrow$, then $\Psi(f_n(x)) = B(f_n (x)) = A(x)$, 
  and hence $\Psi \circ f_n$ agrees with the characteristic
  function of $A_n$ on its domain $D$.  It remains to show that $D$
  has density $1$.  Since $\Psi$ is a generic description, its domain
  $\widehat{D}$ has density $1$.  The increasing bijection $f_n$ from
  $\omega$ to $R_n$ is also an increasing bijection from $D$ to $\widehat{D}
  \cap R_n$.  To show that $D$ has density $1$, it thus suffices to
  show that $\widehat{D} \cap R_n$ has density $1$ in $R_n$.  This
  follows from the general fact that if $C$ is any generic set and $E$
  is any set of positive density, then $C \cap E$ is generic in $E$.
  (Just check that $E \setminus C$ is negligible in E.)
\end{proof}

We do not  know however, whether every generic degree
  bounds only countably many generic degrees.

  The Turing degrees can be embedded into the enumeration degrees by
  the mapping which takes the Turing degree of a set $A$ to the
  enumeration degree of $\gamma(\chi_A)$.  We now give an analogous
  embedding of the Turing degrees into the generic degrees.

\begin{lem} $A \le_T B$ if and only if $\mathcal{R}(A) \leq_g
  \mathcal{R}(B)$.
\end{lem}

\begin{proof} 
  If $\mathcal{R}(A) \leq_g \mathcal{R}(B)$ then
  $\mathcal{R}(A)$ is generically computable from a generic listing of
  $\mathcal{R}(B)$ and thus computable from $B$.  But a generic
  computation of $\mathcal{R}(A)$ allows one to compute $A$. Hence $A
  \le_T B$.  A generic listing of $\mathcal{R}(B)$ allows one to
  compute  $B$ uniformly.   Hence if $A \le_T B$ then
  $\mathcal{R}(A)$ is uniformly computable from any generic listing of
  $\mathcal{R}(B)$ and we have $\mathcal{R}(A) \leq_g
  \mathcal{R}(B)$.
\end{proof}

So $A \equiv_{T} B$ if and only if $\mathcal{R}(A) \equiv_g
\mathcal{R}(B)$, and if $\boldsymbol{a}$ is a Turing degree then
$\deg_g(\mathcal{R}(\boldsymbol{a}))$, defined as
$\deg_g(\mathcal{R}(A))$ for $A \in \boldsymbol{a}$, is
well-defined.

\begin{thm} Let $( \mathbf{D}, \le_T )$ be the set of all Turing
  degrees partially ordered by Turing reducibility and let $(
  \mathbf{I}, \leq_g )$ be the set of all generic degrees
  partially ordered by generic reducibility.  Then the function
  $\mathfrak{B}$ from $\mathbf{D}$ to $\mathbf{I}$ defined by
  $\boldsymbol{a} \mapsto \deg_g(\mathcal{R}(\boldsymbol{a}))$ is
  an order embedding.
\end{thm}

\begin{proof} The remarks above show that $\mathfrak{B}$ is well-defined,
  $1-1$, and order-preserving.
\end{proof}

It follows at once from the above observation and the existence of an
antichain of Turing degrees of the size of the continuum
(\cite{Sacks}, Chapter 2) that there is an antichain of generic
degrees of the size of the continuum.

\begin{thm} The order embedding $\mathfrak{B}$ from the Turing degrees
  to the generic degrees defined above is not surjective.
\end{thm}

We must show that there is a set $A$ such that there is no set $B$
with $A \equiv_g \mathcal{R}(B)$.  Our first task is to give
conditions on $A$ (without mentioning any other sets) which imply that
there is no set $B$ with $A \equiv_g \mathcal{R}(B)$.  These
conditions involve enumeration reducibility, for which we follow Cooper
(\cite{Cooper}, Sections 11.1 and 11.3).  An enumeration degree $\bf a$
is called \emph{total} if there is a total function $f$ such that its
graph $\gamma(f)$ has degree $\bf a$. An enumeration degree $\bf a$ is
called \emph{quasi-minimal} if $\bf a$ is nonzero and every nonzero
enumeration degree ${\bf b \leq_e \bf a}$ is not total.  Thus, a set
$A$ has quasi-minimal enumeration degree if and only if $A$ is not
c.e.~and every total function $f$ with $\gamma(f) \leq_e A$ is
computable.  The next lemma gives the desired conditions on $A$.

\begin{lem} Suppose that $A$ is a set of density $1$ such that $A$
  is not generically computable and the enumeration degree of $A$ is
  quasi-minimal.  Then there is no set $B$ such that $A \equiv_g
  \mathcal{R}(B)$.
\end{lem}

\begin{proof} Suppose for a contradiction that $A$ satisfies the above
  hypotheses and $A \equiv_g \mathcal{R}(B)$.  Let $S_A$ be the
  semicharacteristic function of $A$, that is,  $S_A(n) = 1$ if $n \in
  G$ and  $S_A(n)$ is undefined otherwise.  Note that $A \equiv_e
  \gamma(S_A)$, and $S_A$ is a generic description of $A$ since $A$ has
  density $1$.  Since $\mathcal{R}(B) \leq_g A$, by the definition of
  generic reducibility, there is a generic description $\Theta$ of
  $\mathcal{R}(B)$ such that $\gamma(\Theta) \leq_e S_A$.  However, as
  we have noted, $B$ is computable by a fixed oracle machine from any
  generic description of $\mathcal{R}(B)$.  Hence, if $\chi_B$ is
the characteristic function of $B$, we have 
 \[ \gamma(\chi_B) \leq_e \gamma(\Theta) \leq_e S_A \leq_e A. \]
   Therefore, $\gamma(\chi_B) \leq_e A$.  Since the
  enumeration degree of $A$ is quasi-minimal and $\chi_B$ is total, it
  follows that $\chi_B$ and hence $B$ and $\mathcal{R}(B)$ are
  computable.  As $A \leq_g \mathcal{R}(B)$, we conclude that $A$ is
  generically computable, which is the desired contradiction.
\end{proof}

\begin{proof} To prove the theorem  must now construct a set $A$ satisfying 
the hypotheses of the
above lemma.  We  use a modified version of Cooper's elegant
exposition of Medvedev's proof of the existence of
 quasi-minimal e-degrees. (\cite{Cooper}, Theorem 11.4.2).  In order
to ensure that $A$ has density $1$ we meet the following positive
requirements ensuring that $A$ has density $1$:
$$P_n : R_n \subseteq^* A$$

  In order to ensure that $A$
is not generically computable, we satisfy the following requirements:

$$S_n : \quad \Phi_n \mbox{ does not generically compute } A$$ 

 Note that meeting all the requirements $P_n$ and
$S_n$ ensures that $A$ is not c.e.~since any c.e.~set
of density $1$ is generically computable.

Hence, in order to ensure that the e-degree of $A$ is quasi-minimal, it 
suffices to ensure that every total function $f$ with $\gamma(f) \leq_e
A$ is computable.   Our standard listing $\{W_e\}$ gives us a listing
of enumeration operators.  We will meet the following requirements:

$$ U_n : \mbox{~If~} W_n(A) = \gamma(f) \mbox{ where $f$ is a total function,
  then $f$ is computable}$$ 

We identify partial functions with their graphs.  For example, if
$\theta$ and $\mu$ are partial functions, then $\theta \supseteq \mu$
means that the graph of $\theta$ contains the graph of $\mu$.  We say
that $\theta$ and $\mu$ are \emph{compatible} if they agree on the
intersection of their domains, or, equivalently, $\theta \cup \mu$ is
a partial function.  A \emph{string} is a $\{0,1\}$-valued partial
function $\sigma$ whose domain is equal to $\{0,1, \dots, k-1\}$ for
some $k$ called the \emph{length} of $\sigma$.

At each stage $s$ in the construction of $A$, we will have a partial
computable function $\theta_s$ (taking values in $\{0,1\}$) which
represents the part of the characteristic function of $A$ constructed
by the beginning of stage $s$.  We will have $\theta_{s+1} \supseteq
\theta_s$ for all $s$, and the characteristic function of $A$ will be
$\cup_s \theta_s$.  The domain of $\theta_s$ will be a computable set
having at most finitely elements not in $\cup_{i<s} R_i$.  Further,
there will be only finitely many $x$ with $\theta_s(x) = 0$.  Let
$\theta_0$ be the empty partial function.

If $s = 3n$, then define $\theta_{s+1} \supseteq \theta_s$
by setting $\theta_{s+1} (x) = 1$ for all $x \in \cup_{i<s} R_i$
such that $\theta_s(x) \neq 0$.  These steps will ensure that $\cup_s
\theta_s$ is total and each $R_i \subseteq^* A$.

 If $s = 3n + 1$ we diagonalize against $\Phi_n$.  If there exists an
 $x \in R_s \diagdown dom(\sigma_s)$ with $\Phi_n (x)$ defined then
 let $\sigma_{s+1}(x)$ have a value of $0$ or $1$ which is different
 from $\Phi_n(x)$.  If no such $x$ exists let $\sigma_{s+1} =
 \sigma_s$.  This ensures that the requirement $S_n$ is met because
 $R_s \cap dom(\theta_s)$ is finite and $R_s$ has positive density.

 If $\theta$ is a partial function, let $\theta^{-1}(1) = \{ x :
 \theta(x) = 1 \}$.

  If $s = 3n+2$, there are two cases.  

  {\bf Case 1.}  There exists a string $\sigma_s$ compatible with
  $\theta_s$ and numbers $x$, $y_1$, and $y_2$ such that $y_1 \neq y_2$
  and $\langle x, y_1 \rangle, \langle x, y_2 \rangle \in
  W_n(\sigma_s^{-1}(1))$.

  In this case, let $\theta_{s+1} = \theta_s \cup \sigma_s$,
  ensuring that $W_n(A)$ is not a single valued function.

  {\bf Case 2}.  Otherwise. Let $\theta_{s+1} = \theta_s$.  We must
  show that the requirement $U_n$ is met in this case.  Suppose that
  $W_n(A) = \gamma (f)$ where $f$ is a total function.  We must show
  that $f$ is computable.  Note that the set of strings compatible
  with $\theta_s$ is computable for fixed $s$.  Given $x$, to compute
  $f(x)$ effectively, search effectively for a number $y$ and a string
  $\sigma$ which is compatible with $\theta_s$ such that $\langle x, y
  \rangle \in W_n(\sigma^{-1}(1))$.  We claim that such $\sigma, y$
  exist, and the only possible value for $y$ is $f(x)$, which suffices
  to show that $f$ is computable.  First, observe that there is a
  string $\sigma$ compatible with $\theta_s$ with $\langle x, f(x)
  \rangle \in W_n(\sigma^{-1}(1))$ since $\langle x,f(x) \rangle \in
  W_n(A)$ and $A \supseteq \theta_s$.  Thus, the desired $\sigma$ and
  $y$ exist, in fact with $y = f(x)$.  It remains to show that if
  $\langle x, y\rangle \in W_n(\tau^{-1}(1)$ where $\tau$ is a string
  compatible with $\theta_s$, then $y = f(x)$.  Let $\mu$ be a string
  compatible with $\theta_s$ such that $\mu^{-1}(1) \supseteq
  \sigma^{-1}(1) \cup \tau^{-1}(1)$.  (To obtain $\mu$, let $b$ be the
  greater of the length of $\sigma$ and the length of $\tau$, and, for
  $x < b$, set $\mu(x) = \theta_s(x)$ if $x$ is in the domain of
  $\theta_s(x)$, and otherwise let $\mu(x) = 1$.)  Then, by the
  monotonicity of enumeration operators, $\langle x, f(x) \rangle$ and
  $\langle x, y \rangle$ both belong to $W_n(\tau^{-1}(1))$.  Since
  Case 1 does not apply, we conclude that $y = f(x)$, which completes
  the proof.
\end{proof}

\section{Further results and open questions}

The authors, in ongoing joint work with Rod Downey, have obtained
further results in the area and are working on open questions.  The
section is a brief update on this project.  Full results and proofs
will appear in a later paper \cite{DJS}.

One aspect of the project is the study of the connection between
computability theory and asymptotic density.  Recall that it was shown
in Theorem \ref{nosubset} that there is a c.e.~set $A$ of density $1$
which has no computable subset of density $1$.  In that proof, the
positive requirements $R_n \subseteq^* A$ had an infinitary nature,
and this makes one suspect that no such $A$ is low.  (A set $A$ is
called \emph{low} if $A' \leq_T 0'$ or, in other words, every
$A$-c.e.~set is computable from the halting problem.)  Indeed this is
the case, and we also show that every nonlow c.e.~set computes such an $A$.

\begin{thm} \cite{DJS} \label{low}  The following are equivalent for
  any c.e.~degree $\bf a$:
\begin{enumerate}
    \item  The degree $\bf a$ is not low.
     \item There is a c.e.~set $A$ of degree $\bf a$ such that $A$ has
  density $1$ but no computable subset of $A$ has density $1$.
\end{enumerate}
\end{thm}

Another line of results related to Theorem \ref{nosubset} involves
weakening the requirement that the subsets have density $1$.  The
following result is easy.

\begin{thm} \cite{DJS} If $A$ is a c.e.~set of upper density at least
  $r$, where $r$ is a computable real, then $A$ has a computable
  subset of upper density at least $r$.  In particular, every c.e.~set
  of upper density $1$ has a computable subset of upper density $1$.
\end{thm}

We know by Theorem \ref{nosubset} that this result fails for lower
density even in the case $r=1$, but we show that a slightly weaker
version holds for lower density.

\begin{thm} \cite{DJS} If $A$ is a c.e.~set and $r$ is a real number,
  and the lower density of $A$ is at least $r$, then for each
  $\epsilon > 0$ $A$ has a computable subset whose lower density at
  least $r - \epsilon$.  In particular, every c.e.~set of density $1$
  has computable subsets of lower density arbitrarily close to $1$.
\end{thm}

In Theorem \ref{d02} we showed that the densities of computable sets
are precisely the $\Delta^0_2$ reals in $[0,1]$.  In \cite{DJS} we
 consider analogous results for upper and lower densities, and for
c.e.~sets.  Call a real number $r$ \emph{left}-$\Pi^0_n$ if $\{q \in
\mathbb{Q} : q < r\}$ is $\Pi^0_n$, i.e. the lower cut of $r$ in the
rationals is $\Pi^0_n$.  An analogous definition holds  for other levels
of the arithmetic hierarchy.

\begin{thm} \cite{DJS}.  Let $r$ be a real number in the interval $[0,1]$.
   \begin{enumerate}
        \item $r$ is the lower density of a computable set if and only if
         $r$ is left $\Sigma^0_2$
         \item $r$ is the upper density of a computable set if and only if
         $r$ is left $\Pi^0_2$
       \item $r$ is the density of a c.e.~set if and only if $r$ is
         left $\Pi^0_2$
       \item $r$ is the lower density of a c.e.~set if and only if $r$
         is left $\Sigma^0_3$
       \item $r$ is the upper density of a c.e.~set if and only if $r$
         is left $\Pi^0_2$
     \end{enumerate}
\end{thm}

The other main topic of our ongoing project with Downey is the structure of the
generic degrees and the generic classes.  However, here we are have
not yet been able to answer some questions which would seem to be
basic.

\begin{question} Do there exist noncomputable sets $A,B$ whose generic
  classes form a minimal pair in the sense that every set generically
  computable from both $A$ and $B$ is generically computable?  (That
  is, $\widehat{G}(A) \cap \widehat{G}(B) = \widehat{G}(\emptyset)$.)
\end{question}

So far, our results on the above question have a negative character.
A set $A$ is \emph{hyperimmune} if $A$ is infinite and for every
computable sequence $\{F_i\}_{i \in \omega}$ of pairwise disjoint
finite sets, $A \cap F_i = \emptyset$ for some index $i$.

\begin{thm} \cite{DJS} Let $A$ and $B$ be sets such that $A \cup B$ is
  hyperimmune.  Then $A$ and $B$ do not form a minimal pair in the
  sense of the above question.
\end{thm}

This result shows that minimal pairs for relative generic
computability (if they exist at all) are far rarer than for Turing
reducibility.

\begin{cor} \cite{DJS} The set of pairs $(A,B)$ such that $\widehat{G}(A)
 \cap  \widehat{G}(B) = \widehat{G}(\emptyset)$ is meager and of measure
  $0$ in $2^\omega \times 2^\omega$.
\end{cor}

\begin{cor} \cite{DJS} If $A$ and $B$ are $\Delta^0_2$ sets, then $A, B$ do not
  form a minimal pair for relative generic computability in the above
  sense.
\end{cor}

Further, we do not know whether there exist minimal degrees or minimal
pairs for generic reducibility.  Since there exist hyperimmune sets of
minimal Turing degree, the following result shows that our embedding
$\mathfrak{B}$ of the Turing degrees into the generic degrees need not
map minimal Turing degrees to minimal generic degrees.

\begin{thm} \cite{DJS} If $\bf a$ is a Turing degree and $\bf
  a$ contains a hyperimmune set, then $\mathcal{R}({\bf a})$ is not a minimal
  generic degree.
\end{thm}


\begin{thebibliography}{ABC}

\bibitem{BG} A. Blass, Y. Gurevich,
\emph{Matrix transformation is complete for the average case},
SIAM J. Computing, \textbf{22} (1995), 3-29.

\bibitem{Cooper} S. Barry Cooper,
\emph{Computability Theory},  Chapman and  Hall/CRC, 2004.

\bibitem{DJS} Rod Downey, Carl G.~Jockusch, Jr., and Paul Schupp,
  \emph{Asymptotic density and computably enumerable sets} (tentative
  title), in preparation.

\bibitem{Gurevich} Y.~Gurevich, \emph{Average case completeness},
J. of Computer and System Science  \textbf{42} (1991), 346--398.

\bibitem{Klee}
V.~Klee and G.~Minty,
\emph{How good is the simplex algorithm?}
Inequalities, III (Proc. Third Sympos., Univ. California, Los Angeles,
Calif., 1969; dedicated to the memory of Theodore S. Motzkin), pp.
159--175. Academic Press, New York, 1972.

\bibitem{KMSS} Ilya Kapovich, Alexei Miasnikov, Paul Schupp, and
Vladimir Shpilrain, \emph{Generic-case complexity, decision problems
in group theory and random walks}, J. Algebra \textbf{264} (2003), 665-694.


\bibitem{KS1} Ilya Kapovich and Paul Schupp,
\emph{Genericity, the Arshantseva-Ol'shanskii technique
and the isomorphism problem for one-relator groups}, Math. Annalen,
\textbf{331} (2005), 1-19.


\bibitem{KSS} Ilya Kapovich, Paul Schupp, and Vladimir Shpilrain,
\emph{Generic properties of Whitehead's Algorithm and isomorphism
rigidity of one-relator groups}, Pacific Journal of Mathematics, 
\textbf{223} (2006), 113-140


\bibitem{Levin} L.~Levin, \emph{Average case complete problems}, SIAM
Journal of Computing \textbf{15} (1986), 285--286.


\bibitem{L-S}
R.~C.~Lyndon and P.~E.~Schupp, \emph{Combinatorial Group Theory},
Ergebnisse  der Mathematik, Band 89, Springer 1977.  Reprinted
in the Springer Classics in Mathematics series, 2000.


\bibitem{Magnus}
W.~Magnus, \emph{ Das Identit\"atsproblem f\"ur Gruppen mit einer
definierenden Relation}, Math. Ann., \textbf{106} (1932), 295--307.


\bibitem{MR}
A.~Miasnikov and A.~Rybalov, \emph{Generic complexity of undecidable problems},
Journal of Symbolic Logic, \textbf{73} (2008), 656-673.


\bibitem{MM}
W.~Miller and D.~A.~Martin, \emph{The degrees of hyperimmune sets},
Zeitschr.~f.~math.~Logik und Grundlagen d.~Math. \textbf{14} (1968), 159--165.


\bibitem{Nies} A.~Nies, \emph{Computability and Randomness}, Oxford
  Logic Guides 51, Oxford University Press, Oxford, New York, 2009.

\bibitem{Papa}
C.~Papadimitriou, \emph{Computational  Complexity}, (1994),
Addison-Wesley, Reading.

\bibitem{Rotman}
J.~Rotman,
\emph{An Introduction to the Theory of Groups}, Fourth Edition,
Graduate Texts in Mathematics, \textbf{148}, Springer-Verlag, New York,
1995.

\bibitem{Sacks}
G.E.~Sacks,
 \emph{Degrees of Unsolvability}, Second Edition,
  Annals of Mathematics Studies, No.~55, Princeton University Press,
  Princeton, New Jersey, 1966.


\bibitem{Wang97}
J.~Wang, \emph{Average-case computational complexity theory},
Complexity Theory Retrospective, II. Springer-Verlag, New York,
1997, 295--334.


\bibitem{Woess}
W.~Woess, \emph{Cogrowth of groups and simple random walks},
Arch. Math.   { \bf 41} (1983),   363--370.




\bibitem{Wo}
W.~Woess, \emph{Random walks on infinite graphs and groups - a
survey on selected topics,} Bull. London Math. Soc. \textbf{26}
(1994), 1--60.

\end{thebibliography}
\end{document}